\documentclass[11pt, reqno, a4paper]{amsart} 
\usepackage[english]{babel}
\usepackage{amsfonts, amsmath, amsthm, amssymb,amscd,indentfirst}
\usepackage{times}
\usepackage{enumerate}
\usepackage{anysize} 

\marginsize{1.8cm}{1.8cm}{1.5cm}{1.5cm}

\newtheorem{theorem}{Theorem}[section]
\newtheorem{proposition}{Proposition}[section]
\newtheorem{lemma}{Lemma}[section]
\newtheorem{definition}{Definition}[section]

\newtheorem{condt}{Assumption}[section]
\newtheorem*{rem}{Remark}

\newcommand{\loc}{\operatorname{loc}}
\newcommand{\di}{d}
\newcommand{\ao}{A}
\newcommand{\bo}{B}
\newcommand{\esup}{\operatorname{ess} \operatorname{sup}}
\newcommand{\einf}{\operatorname{ess} \operatorname{inf}}
\newcommand{\vl}{\vartheta}
\newcommand{\spt}{\operatorname{spt}}
\newcommand{\inter}{\operatorname{int}}
\newcommand{\dvt}{\textnormal{d}\vartheta}
\newcommand{\dx}{\textnormal{d}x}
\newtheorem{thrm}{Theorem}

\title{Removable singularities for sub-solutions of  elliptic equations involving  weighted variable exponent Sobolev spaces}
\author{Juan Alcon Apaza}

\address{Universidade Federal Fluminense, Instituto de Matemática, Campus do Gragoatá, Rua Prof. Marcos Waldemar de Freitas, s/n, bloco H, Niterói, RJ 24210-201, Brazil }
\email{jpablo@id.uff.br}

\begin{document}

\maketitle

\begin{abstract}
We study the removability of a singular set for elliptic equations involving weight functions and variable exponents. We consider the case where the singular set satisfies conditions related to some generalization of \textit{upper Minkowski content} or a \textit{net measure}, and give sufficient conditions for removability of this singularity for  equations in the weighted variable exponent Sobolev spaces $W^{1,p(\cdot)}(\Omega,\vl)$.
\end{abstract}

\let\thefootnote\relax\footnote{2020 \textit{Mathematics Subject Classification}. 35A21,  35J57, 35J60.}
\let\thefootnote\relax\footnote{\textit{Key words and phrases}. variable exponent, weight function, singular set, removable singularity.}

\markright{REMOVABLE SINGULARITIES FOR SUB-SOLUTIONS OF  ELLIPTIC EQUATIONS}

\section{Introduction}

This paper is devoted to the study of conditions guaranteeing the removability of singular set for non-negative weak sub-solutions, in the weighted variable exponent Sobolev spaces, for nonlinear elliptic equations of the form
$$
-\operatorname{div} \ao (x,u,\nabla u) + B(x,u)=0.
$$
For the exacts requirements on the mappings $\ao$ and $\bo$ we refer to Section \ref{66}. In this work  we employ the arguments used to prove the following results.

\begin{thrm}(see \cite[Theorem 1]{rii}) \label{68}
Let  $\Omega\subset \mathbb{R}^n$, $n\geqslant 2$, be an open set, and $E\subset \Omega$ be closed  in $\Omega$. Denote by $\mathcal{M} ^\alpha (E)$ the $\alpha$-dimensional upper Minkowski content of $E$. Suppose that $\alpha \in [0,n-2]$ and $\mathcal{M} ^\alpha (E)=0$. If $f$ is subharmonic in $\Omega \backslash E$ and satisfies 
$$
f(x) \leqslant C \di (x,E) ^{\alpha + 2-n} \quad \left(x\in \Omega \backslash E\right)
$$  
for some positive constant $C$, then $f$ has a subharmonic extension to $\Omega$.
\end{thrm}

\begin{thrm}(see \cite[Theorem 3]{rii})\label{69}
Let  $\Omega\subset \mathbb{R}^n$, $n\geqslant 2$, be an open set, and $E\subset \Omega$ be closed  in $\Omega$. Denote by $\mathcal{H} ^\alpha (E)$ the $\alpha$-dimensional Hausdorff measure of $E$. Suppose that $\alpha \in [0,n-2]$ and $\mathcal{H} ^\alpha (E)=0$. If $f$ is subharmonic in $\Omega \backslash E$ and satisfies 
$$
\frac{1}{\left|B(x,r) \right|} \int _{B(x,r)} \max \{f,0\} \dx  \leqslant C r^{\alpha + 2 -n} \quad \left( \overline{B(x,r)} \subset \Omega\right)
$$  
for some positive constant $C$, then $f$ has a subharmonic extension to $\Omega$.
\end{thrm}

\noindent  Motivated by these theorems we given a generalization of the \textit{upper Minkowski content} and the \textit{net measure} involving weight and variable exponents, see the next section. In the case that certain smoothness, Hausdorff measure or uniform Minkowski condition is imposing on the singular set, a different approach can be found in \cite{gar, hir, kos}.

\section{Preliminaries}

Throughout the whole article  $\Omega$ is a bounded open set in $\mathbb{R}^n$, $n\geqslant 3$. We will work on $\Omega$ with the Lebesgue measure $\dx$. We first  recall some facts on spaces $L^{ p(\cdot)} (\Omega , \vl)$ and $W^{1 , p(\cdot)} (\Omega  , \vl)$. We write
$$
L^\infty _+ (\Omega)= \left\{p\in L^\infty (\Omega) \:|\: \einf _\Omega p > 1\right\}.
$$ 
For $p\in L^\infty _+ (\Omega)$, we define
$$
W(\Omega) = \left\{\vl \in L^1 _{\loc} (\Omega) \:|\: \vl > 0 \text { almost everywhere in } \Omega\right\}.
$$
An element in $W(\Omega) $ is called a \textit{weight function}. For $p\in L^\infty _+ (\Omega)$, set
$$
W_p (\Omega) = \left\{ \vl \in W(\Omega) \:|\: \vl ^{-\frac{1}{p-1}} \in L^1 _{\loc} (\Omega)\right\}.
$$ 

Let $p\in L^\infty _+ (\Omega)$ and $\vl \in W(\Omega)$, we define the functional 
$$
\rho  _{p ,\vl}(u) = \int _{\Omega } |u|^{p} \vl \dx.
$$

The \textit{weighted variable exponent Lebesgue space} $L^{p(\cdot)}(\Omega , \vl)$ is the class of all functions $u$ such that $\rho _{p ,\vl}(t u)<\infty$, for some $t>0$. $L^{p(\cdot)}(\Omega , \vl)$ is a Banach space equipped with the norm
$$
\|u\|_{L^{p(\cdot)} (\Omega , \vl)}=\inf \left\{\lambda>0\:|\:  \rho _{p ,\vl}\left(\frac{u}{\lambda}\right) \leq 1\right\};
$$
see \cite[Theorem 2.5]{kova}. In the case that $\vl \equiv 1$, we have $L^{p(\cdot)} (\Omega , \vl) = L^{p(\cdot)} (\Omega)$. 

For any $p\in L^\infty _+ (\Omega)$, we denote by  $p^\prime = \frac{p}{p-1}$ the conjugate function.

\begin{proposition} (see \cite[Theorem 2.1]{kova}.)
Let $p \in L^\infty _+ (\Omega)$. For any $u \in L^{p(\cdot)}(\Omega)$ and $v \in L^{p^{\prime}(\cdot)}(\Omega)$,
\begin{equation}\label{31}
\int_{\Omega}|u v| \dx\leq 2\|u\|_{L^{p(\cdot)} (\Omega)}\|v\|_{L^{p^\prime(\cdot)} (\Omega)} .
\end{equation}
\end{proposition}
\begin{proposition}\label{32}  (see \cite[pp. 4]{una}, \cite[Theorem 1.3]{fan}.)
Let $p \in L^\infty _+ (\Omega)$. For any $u \in L^{p(\cdot)}(\Omega , \vl)$, we have
\begin{gather} 
\min \left\{  \rho _{p , \vl} (u) ^{\frac{1}{p^-}} ,  \rho _{p , \vl} (u) ^{\frac{1}{p^+}} \right\}\leqslant \left\|u \right\| _{L^{p(\cdot)} (\Omega , \vl)} \leqslant \max \left\{   \rho  _{p , \vl} (u) ^{\frac{1}{p^-}} ,  \rho _{p , \vl} (u) ^{\frac{1}{p^+}} \right\}, \label{57}\\ 
\min \left\{ \left\|u \right\| _{L^{p(\cdot)} (\Omega , \vl)} ^{p^-} ,  \left\|u \right\| _{L^{p(\cdot)} (\Omega , \vl)} ^{p^+} \right\}\leqslant  \rho _{p , \vl} (u) \leqslant \max \left\{ \left\|u \right\| _{L^{p(\cdot)} (\Omega , \vl)} ^{p^-} ,  \left\|u \right\| _{L^{p(\cdot)} (\Omega , \vl)} ^{p^+} \right\},\label{58}
\end{gather}
where $p^- = \einf _{\Omega} p$ and $p^+ = \esup _{\Omega} p$.
\end{proposition} 

Let $p\in L^\infty _+  \left( \Omega \right)$ and  $\vl \in W_p (\Omega)$. The \textit{weighted variable exponent Sobolev space} $W^{1 , p(\cdot)} \left(\Omega  , \vl\right)$ is the class of all functions $u \in L^{p(\cdot)}(\Omega , \vl)$ which have the property $|\nabla u| \in L^{p(\cdot)}(\Omega ,\vl)$.  The space $W^{1  , p(\cdot)} \left(\Omega , \vl\right)$  is a Banach space equipped with the norm
$$
\|u\|_{W^{1,  p(\cdot)} (\Omega , \vl)}=\|u\|_{L^{p(\cdot)} (\Omega , \vl)}+\|\nabla u\|_{L^{p(\cdot)} (\Omega ,  \vl)};
$$
see \cite[Theorem 3.1]{kova} and \cite[Proposition 3.2]{una}.

In the proof our main results we employed the following type of partition of unity:

 \begin{lemma}(see \cite[Lemma 3.1]{harvey}) \label{26}
 Let $\{Q_i\}_{i=1} ^N$ be a finite disjoint collection of dyadic of length $\ell _i$. For each $i$, there is a non-negative function $\varphi _i \in C_0 ^\infty (\mathbb{R}^n)$ with $\spt \varphi _i \subset  \frac{3}{2} Q_i $ such that  $\sum _{i=1} ^N \varphi _i  \leqslant 1$ on $\mathbb{R}^n$, and $\sum _{i=1} ^N \varphi _i  = 1$ on $ \cup _{i=1} ^N Q_i$. Furthermore, for each multi-index $\lambda$, there is a constant $C_\lambda = C_\lambda (n,\lambda)>0$ for which $|D^\lambda \varphi _i (x)| \leqslant C_\lambda \ell _i ^{-|\lambda|}$ for all  $x$ and $i\in \{1,\ldots ,N\}$.
 \end{lemma}

To define the \textit{Minkowski content}, let  $\alpha : \Omega \rightarrow [0,n]$ be a measurable function and $F\subset \mathbb{R}^n$. For $\varepsilon > 0$, we write
$$
F_\varepsilon = \{x\in \mathbb{R}^n \:|\: \di ( x , F )< \varepsilon\},
$$
where  $\di (x, F)$ denote the Euclidean distance from a point $x \in \mathbb{R}^n$ to $F$.

\begin{definition}
The $(\alpha , \vl )$-\textnormal{upper Minkowski content} of $F$ is defined by
$$
\mathcal{M} ^\alpha  _{\vl} (F) = \limsup _{\varepsilon \rightarrow 0^+} \int _{F_\varepsilon} \varepsilon^{\alpha-n}\dvt,
$$
where $\vl \in W(\Omega)$ and  $\dvt = \vl \dx$. 
\end{definition}

Now, we proceed to define the \textit{net measure}. For $x\in \mathbb{R}^n$ and $\ell >0$, we denote by
$$
Q(x,\ell)=\left\{ y \in \mathbb{R} ^n \:|\: |x_i-y_i|\leqslant \ell /2, \ i=1, \cdots,n\right\},
$$
the cube of center $x$ and length $\ell$. Recall that a \textit{dyadic cube} is represented as
\begin{equation*} \label{29}
\left[m_1 2^{-k} , \left(m_1 +1 \right)2^{-k}\right]\times \cdots \times \left[m_n 2^{-k} , \left(m_n +1 \right)2^{-k}\right],
\end{equation*}
where $k, m_1, \ldots, m_n $ are integers.   

For a cube  $Q$  of length $\ell (Q) = \ell$, we let $\frac{3}{2}Q$ denote the cube with the same center and length $3\ell /2$.  For any integer $k$ let $G_k$ be a collection of cubes represented by a form 
$$
\left[m_1 2^{-k} , (m_1 +1)2^{-k}\right]\times \cdots \times \left[m_n 2^{-k} , (m_n +1)2^{-k}\right],
$$
for some integers $m_1, \ldots, m_n$. 

Let $\gamma >0$ be  fixed. For  each $\varepsilon >0$ we write
\begin{equation*}
\mathcal{L}^\varepsilon _{\alpha , \gamma , \vl} (F) = \inf \left\{ \sum _{j} \left[\int _{Q_j} \left(2^{-k_j} \right)^{\alpha - n} \dvt \right]^{\gamma} \:|\:  F\subset \inter \left( \cup _j Q_j\right), \
Q_j \in G_{k_j}, \ 2^{-k_j} < \varepsilon  \right\} . 
\end{equation*}

\begin{definition}
The $(\alpha , \gamma , \vl)$-\textnormal{net measure} of $F$ is defined by
$$
\mathcal{L} _{\alpha , \gamma , \vl} (F) = \lim _{\varepsilon \rightarrow 0^+} \mathcal{L}^\varepsilon _{\alpha , \gamma , \vl} (F).
$$
\end{definition}

We finish this section, introducing the following condition used in this work.

\begin{condt} \label{70}
Assume $p\in L^\infty _+ (\Omega)$ and $\vl\in W _p (\Omega)$. Suppose   $\varphi \in W^{1,p(\cdot)} (\Omega , \vl)$ has compact support $\spt \varphi =K$. Then there exist functions $\varphi _m \in C^{\infty} (\Omega)$, with $\spt \varphi _m \subset K_{\di (K,\partial \Omega) /2}$, such that
$$
\varphi _m \rightarrow \varphi \quad \text { in }  W^{1,p(\cdot)} (\Omega , \vl).
$$
\end{condt}

\begin{rem}
In the case that $\vl \equiv 1$ and $p\in \mathcal{P} ^{\log} (\Omega)$ is bounded, the Assumption \ref{70} is verified, see \cite[Corollary 11.2.4]{dien}. On the other hand, if $p\in \mathcal{P} ^{\log} (\Omega)$ is bounded and $\vl$ satisfies the  Muckenhoupt condition, see \cite{cruz, dien1}, following the same lines of the proof  of \cite[Lemma 2.4]{mill} we obtain the Assumption \ref{70}; for auxiliary results see \cite[Theorem 2.1.4]{tur}, \cite[Appendix C.5]{evans}, \cite[Theorem 1.1]{cruz} and \cite[Theorem 1.1]{dien1}. Finally, if $p\equiv \operatorname{constant}$ and $\vl$ belongs to Muckenhoupt's $A_p$-class,  the Assumption \ref{70} is satisfied;  see \cite[Theorem 2.5]{kil}.
\end{rem}

\section{Main results} \label{66}

In the sequel $E\subset \Omega$ is a compact set. For us a function is \textit{integrable} if it has an integral (which may equal $\pm \infty$).
\begin{definition}
Let $\vl \in W_p (\Omega)$ and  $p \in L^\infty _+ (\Omega)$. We define
\begin{multline*}
 W^{1, p(\cdot)} _{E,\loc}  \left(\Omega  , \vl \right)=\left\{u:\Omega \rightarrow \overline{\mathbb{R}} \:|\: u \text { is integrable, } u|_{\Omega\backslash E} \in W^{1 , p(\cdot)} _{\loc} (\Omega \backslash E , \vl), \right.\\
 \left.\text { and there is a integrable function }  \widetilde{\nabla u} : \Omega \rightarrow  \overline{\mathbb{R}}^n \text { so that } \widetilde{\nabla u} = \nabla u \text { \textit{a.e.} on }  \Omega\backslash E \right\}.
\end{multline*}
\end{definition}

Now, we will divide our main results into two sections.

\subsection{Singular set satisfying $\mathcal{M} ^\alpha _{\vl} (E)=0$} First, we will work with nonlinear elliptic equations of the form

\begin{equation} \label{1}
-\operatorname{div}  \ao \left( x , u , \nabla u \right) + \bo (x ,u)= 0.
\end{equation}

We assume that $\ao:\Omega \times [0,\infty) \times \mathbb{R}^n \rightarrow \mathbb{R}^n$ and $\bo:\Omega \times [0,\infty) \rightarrow \mathbb{R}$ are measurable functions. Furthermore, there exists $\mu >0$ such that the following conditions are satisfied almost everywhere:
\begin{gather}
\left\langle \ao (x,u,\eta) , \eta \right\rangle \geqslant \left( \mu   |\eta|^{p(x)}   - \mu ^{-1}  u^{p(x)}  + a_1 (x) \right)\vl (x), \label{3}\\
\left|\ao (x,u,\eta)\right| \leqslant  \left(\mu ^{-1} |\eta|^{p(x)-1}  + \mu ^{-1} u ^{p(x) -1}   + a_2 (x) \right) \vl (x), \label{4}\\
|\bo (x,u)| \leqslant \left(\mu ^{-1} u^{p(x) -1}  + a_3 (x) \right)\vl (x), \label{5}
\end{gather}
where  $\vl \in W_p(\Omega)$, $a_1 \in L^{q_1} (\Omega , \vl)$,  $a_2 \in L^{q_2} (\Omega , \vl)$, $a_2\geqslant 0$, $a_3 \in L^{q_3} (\Omega , \vl)$, $a_3\geqslant 0$,  $\alpha, \beta :\Omega \rightarrow [0,n]$ are measurable functions, and $ p,  q_1, q_2, q_3 \in L^\infty _+ (\Omega)$. Furthermore,
\begin{gather}
 \left( p\beta \right) ^+ \leqslant n-\alpha ^+ -p^+ -1, \label{24}\\
 q_1 ^- , \ q_3 ^- \geqslant \left(\frac{p}{p-1}\right)^+, \ q_2 ^- \geqslant \max\left\{ \left(\frac{p}{p-1} \right) ^+ , \frac{n-\alpha ^+}{n-\alpha ^+ -\beta ^+ -2}\right\}. \label{8}
\end{gather}
\begin{definition}
We will say that $u \in W^{1, p(\cdot)} _{E,\loc}  \left(\Omega , \vl \right)$ is a (weak) sub-solution  of equation \eqref{1} in $\Omega \backslash E$ if  
\begin{equation} \label{2}
\begin{split}
 \int_{\Omega \backslash E } \left\langle \ao (\cdot , u , \nabla u) , \nabla \varphi \right\rangle   +   \bo ( \cdot , u ) \varphi  \dx \leqslant 0
\end{split}
\end{equation}
for all  $\varphi \in W^{1, p(\cdot)}\left(\Omega  , \vl \right) $, $\varphi \geqslant 0$, with compact support in $\Omega \backslash E$.
\end{definition}

\begin{definition}
We will say that the sub-solution $u$ of equation \eqref{1} in $\Omega \backslash E$ has a removable singularity at  $E$: if $u \in W^{1 , p(\cdot)} _{E, \loc}  \left(\Omega  , \vl \right)$ implies $u \in W^{1, p(\cdot)} _{\loc}  \left(\Omega ,  \vl \right)$ and the equality \eqref{2} is fulfilled for all $\varphi \in W^{1 , p(\cdot)}  \left(\Omega  , \vl \right)$,  $\varphi \geqslant 0$, with compact support.
\end{definition}

Inspired by  Theorem \ref{68} we have our first main result.
 
\begin{theorem} \label{27}
Suppose that \eqref{3} - \eqref{8} are satisfied, furthermore $\mathcal{M} ^\alpha _{ \vl} (E)=0$. If $u$ is a non-negative sub-solution of \eqref{1} in $\Omega \backslash E$, and 
\begin{equation}\label{25}
u \leqslant C d(\cdot,E)^{-\beta}  \quad \text { \textit{a.e.} on  } \Omega \backslash E,
\end{equation}
where $C$ is a suitable positive constant, then the singularity of $u$ at $E$ is removable. 
\end{theorem}

\begin{proof} 1. Let $\varepsilon \in(0,1)$ be arbitrary. Since $M^{\alpha} _{\vl}(E)=0$,  there is $r_{0} \in(0,\varepsilon)$ such that $\overline{E_{r_0}} \subset \Omega$ and
\begin{equation}\label{11}
\int_{E_{r}} r^{\alpha-n} \dvt<\varepsilon \quad \text { for all }  r \in\left[0, r_{0}\right].
\end{equation}
For $r \in\left[0, r_{0}\right]$ and $j=0,1, \ldots$, we write
$$
K _j=\left\{x \in \mathbb{R}^{n} \:|\: \di (x, E)< 2^{-j} r \right\},
$$
then
$$
E_{r}=\bigcup _{j=0}^{\infty}\left(K_{j} \backslash K_{j+1}\right),
$$
and, by \eqref{25}, \eqref{11} and \eqref{24},
$$
\begin{aligned}
&	\int_{K_{j} \backslash K_{j+1}} u^{p} \dvt  \leqslant \int_{K_{j} \backslash K_{j+1}} C^p \di( \cdot , E)^{-\beta p} \dvt 
	\leqslant C_1 \left(p\right) \int_{K_{j} \backslash K_{j+1}}\left(2^{-(j+1)} r \right)^{-\beta p} \dvt \\  
&	= C_1\int_{K_{j} \backslash K_{j+1}} 2^{\beta p}  \left(2^{-j}r\right)^{n-\alpha-\beta p} \left( 2^{-j} r\right)^{\alpha -n} \dvt 	\leqslant C_1  2^{(\beta p)^{+}}  \left(2^{-j}r\right) ^{n-\alpha^{+}-(\beta p )^{+} } \varepsilon \\
 & \leqslant C_2 (p,\beta)  2^{-j -1 }r^{p^+ +1}\varepsilon .
\end{aligned}
$$
Therefore, considering $|E|=0$,
\begin{equation}\label{10}
\int_{E_{r}}u^{p} \dvt   = \sum_{j=0}^{\infty} \int_{K_{j} \backslash K_{j+1}} u^{p} \dvt \leqslant C_2 \varepsilon  r^{p^+ + 1},
\end{equation}
for all $r \in\left[0, r_{0}\right]$. Thus $u \in L^{p(\cdot)} _{\loc} \left(\Omega, \vl\right)$. \\

2. Now, we prove $|\widetilde{\nabla u}| \in L^{p(\cdot)} _{\loc}\left(\Omega, \vl\right)$. For $r<2 r_{0} /5$, let $\psi_{r}: \mathbb{R} \rightarrow \mathbb{R}$ be a smooth function such that
$$
\psi_{r}(t)= \left\{
\begin{aligned}
&0 & & \text { if } t<r/2 \text { or } t>5r/2, \\
 &1 & & \text { if } r<t<2 r,
\end{aligned}
\right.
$$ 
$0 \leqslant \psi_{r} \leqslant 1$ and $\left|\psi_{r}^{\prime}\right| \leqslant C_3 / r$, where $C_3$ is a positive constant. Set
$$
 \varphi=\left(\psi_{r}^{p^{+}} \circ \di(\cdot, E)\right) u.
$$
We have $\varphi \in W^{1, p(\cdot)}\left(\Omega , \vl \right)$, $\varphi \geqslant 0$, with compact support in $\Omega \backslash E$. For simplicity we write $\psi_{r}=\psi_{r} \circ \di (\cdot, E)$ and $\psi_{r}^{\prime}=\psi_{r}^{\prime} \circ \di (\cdot, E)$. Substituting $\varphi$ into \eqref{2},
$$
\int_{\Omega \backslash E}\left\langle \ao(\cdot, u, \nabla u), \psi_{r}^{p^{+}} \nabla u  + p^{+}  \psi_{r}^{p^{+}-1} \psi_{r}^{\prime} u \nabla \di(\cdot, E)\right\rangle 
+  \bo \left( \cdot , u \right) \psi _r ^{p^+} u \dx \leqslant 0.
$$
By conditions \eqref{3} - \eqref{8},
\begin{equation} \label{67}
 q_1 ^- \geqslant \frac{n-\alpha ^+ }{n-\alpha ^+ - 1 }, \ q_2 ^- \geqslant \frac{n-\alpha ^+}{n-\alpha ^+ -\beta ^+ -2}, \ q_3 ^-\geqslant \frac{n-\alpha ^+}{n-\alpha ^+ - \beta ^+ -1}
\end{equation}
and
$$
\begin{aligned}
&\int_{\left\{ \frac{r}{2} \leqslant \di(\cdot, E) \leqslant \frac{5 r}{2}\right\} }  \psi_{r}^{p^{+}} \left( \mu   |\nabla u|^{p}  - \mu ^{-1} u^p  + a_1   \right)\dvt \\
&\leqslant  \int_{\left\{\frac{r}{2} \leqslant \di( \cdot , E) \leqslant \frac{5 r}{2}\right\}} \frac{C_3}{r} p^{+} \psi_{r}^{p^{+}-1}  \left( \mu^{-1}|\nabla u|^{p-1} + \mu ^{-1} u^{p-1}   + a_2  \right) u \dvt \\
&+\int_{\left\{\frac{r}{2} \leqslant \di( \cdot , E) \leqslant \frac{5 r}{2}\right\}} \psi_{r}^{p^{+}}  \left( \mu^{-1} u^{p-1}  + a_3   \right) u \dvt
\end{aligned}
$$
Consequently, by  \eqref{10},  \eqref{31}, \eqref{57}, Young's inequality and  \eqref{25},
$$
\begin{aligned}
&\int_{\left\{ r \leqslant \di(\cdot , E) \leqslant 2r\right\} } |\nabla u|^{p} \dvt \leqslant C_4 \varepsilon \frac{5r}{2}\\ 
&+ C_4 r^{\left( n-\alpha ^+ \right)\left(\frac{q_1^- - 1}{q_1 ^-} \right)} \|a_1\| _{L^{q_1 (\cdot)} (\Omega , \vl)} \left[\int _{E_{5r / 2}} \left(\frac{5r}{2}\right)^{\alpha-n}\dvt\right] ^{\left(\frac{q_1 -1}{q_1}\right) ^-}\\
& + C_4 \int_{\left\{ \frac{r}{2} \leqslant \di(\cdot , E) \leqslant \frac{5 r}{2}\right\}  } C_5 \left(\varepsilon_{1} , p \right)\left(\frac{u}{r}\right)^{p}  +\varepsilon_{1} \psi_{r}^{\frac{(p^+ -1)p}{p-1}}|\nabla u|^{p} \dvt\\
& + C_4 r^{\left( n-\alpha ^+ \right)\left(\frac{q_2 ^- - 1}{q_2 ^-} \right)-\beta ^+ -1} \|a_2\| _{L^{q_2 (\cdot)} (\Omega , \vl)} \left[\int _{E_{5r /2}} \left(\frac{5r}{2}\right)^{\alpha-n} \dvt\right] ^{\left(\frac{q_2 -1}{q_2}\right) ^-}\\
& + C_4 r^{\left( n-\alpha ^+ \right)\left(\frac{q_3^- - 1}{q_3 ^-} \right) -\beta ^+} \|a_3\| _{L^{q_3 (\cdot)} (\Omega , \vl)} \left[\int _{E_{5r /2}} \left(\frac{5r}{2}\right)^{\alpha-n}\dvt\right] ^{\left(\frac{q_3 -1}{q_3}\right) ^-},
\end{aligned}
$$
where $C_4 = C_4 \left(\mu  , \alpha , p , C_2   \right)>0$. Take $\varepsilon_{1}=\frac{1}{2 C_4}$. From \eqref{67},  \eqref{25}, \eqref{11}, \eqref{24}, we have
\begin{equation*}
\frac{1}{2} \int_{\{r \leqslant \di(\cdot , E) \leqslant 2 r\}}|\nabla u|^{p} \dvt 
\leqslant C_6 \varepsilon ^{\gamma _1} r   + C_6  \varepsilon r^{n-\alpha ^+ - (p\beta )^{+} - 1}  \leqslant 2 C_6  \varepsilon ^{\gamma _1} r.
\end{equation*}
Therefore,
\begin{equation}  \label{14}
\int _{E_r } |\widetilde{\nabla u}|^p \dvt
 = \sum_{i=0}^{\infty} \int_{\left\{ 2^{-(i+1)} r \leqslant \di (\cdot, E)  \leqslant 2^{-i} r\right\}}|\nabla u|^{p} \dvt
 \leqslant  C_7 \varepsilon ^{\gamma _1} r,
\end{equation}
where $C_i = C_i (\mu , \alpha  , \beta , p    , a_1 , a_2 , a_3 )>0$, $i=6,7$, $\gamma _1= \gamma _1 (q_1, q_2 , q_3) \in (0,1)$. Hence  $|\widetilde{\nabla u}| \in L^{p(\cdot)}_{\loc} (\Omega , \vl)$. \\

3. Recall from step 1 and 2 that $u\in L^{p(\cdot)} _{\loc} (\Omega , \vl)$ and  $|\widetilde{\nabla u}|\in L^{p(\cdot)} _{\loc} (\Omega , \vl)$. We now claim $u\in W_{\loc}^{1, p(\cdot)}\left(\Omega, \vl \right)$.  To verify this assertion, fix $\varphi\in C_0 ^{\infty} (\Omega)$, we assume that $|\varphi| \leqslant 1$ and $|\nabla \varphi| \leqslant 1$. For $r\in (0,r_0 /2)$, let $\zeta _r : \mathbb{R} \rightarrow \mathbb{R}$ be a smooth function such that
$$
\zeta _r (t) = \left\{\begin{aligned}
&1 & & \text { if } |t|\leqslant r,\\
&0 & & \text { if } 2r \leqslant |t|,
\end{aligned}
\right.
$$
$0\leqslant \zeta _r \leqslant 1$ and $|\zeta _r ^\prime| \leqslant C_8 /r$, where $C_8$ is a positive constant.  For simplicity we write $\zeta _r = \zeta _r \circ \di (\cdot , E)$ and $\zeta _r ^\prime =  \zeta ^\prime _r \circ \di (\cdot , E)$. We have 
$$
\int _{\Omega \backslash E} u \left[\left(1-\zeta _r\right)\varphi  \right] _{x_i}\dx=- \int _{\Omega \backslash E} u_{x_i}\left(1-\zeta _r\right)\varphi  \dx,
$$
where $i=1, \ldots, n$. If $r\rightarrow 0^+$:
\begin{equation}\label{56}
\int _{\Omega} u \varphi _{x_i}\dx=- \int _{\Omega } u_{x_i}\varphi  \dx.
\end{equation}
Indeed, \eqref{31},  \eqref{57} and \eqref{10}  implies
\begin{align*}
&\left|\int _{\Omega \backslash E} u \left(\zeta _r \right)_{x_i} \varphi \dx\right|\leqslant 2 \left\|\frac{C_8 u}{r} \vl ^{1/p}\right\| _{L^{p(\cdot)} (E_{2r})} \left\|\vl ^{-1/p}\right\|_{L^{p^{\prime}(\cdot)}(E_{2r})} \\
&\leqslant  C_{9}(p,\beta) \left\|\vl ^{-1 / p}\right\|_{L^{p^{\prime}(\cdot)}\left(E_{r_0}\right)} (\varepsilon r )^{1 / p^+} \rightarrow 0 \quad \text { as }  r \rightarrow 0^+,
\end{align*}
since $\vl \in W_{p} (\Omega)$. So \eqref{56} is proved. \\

4. To complete the proof of the theorem. Let $\varphi \in W^{1, p(\cdot)}\left(\Omega, \vl \right)$ be with compact support and $\varphi \geqslant 0$. By  Assumption \ref{70}, we consider the case  $\varphi \in C^{\infty} (\Omega)$,  $\varphi \leqslant 1$ and $|\nabla \varphi| \leqslant 1$. We show
$$
\int_{\Omega}\left\langle \ao ( \cdot , u, \nabla u), \nabla \varphi \right \rangle   + \bo ( \cdot , u) \varphi  \dx\leqslant 0.
$$
Let $K= \spt\varphi$. We may suppose that $\overline{K_{r_{0}} }\subset \Omega$. Choose $\ell = 2^{-k}$ so small that $3 \ell \sqrt{n} < r_{0}$. Cover $K$ by a finite collection of dyadic cubes $Q_i$ with same length $\ell (Q_i)= \ell $, $i=1, \ldots, N$. Moreover, we can assume that 
$$
\begin{aligned}
\frac{3}{2} Q_{i} \cap E &\neq \emptyset  \quad \text { for }  i=1, \ldots, N_{\ast}, \\
\frac{3}{2} Q_{i} \cap E &=\emptyset \quad \text { for }  i=N^{\ast}+1, \ldots, N,
\end{aligned}
$$
for some $N_{\ast} \in \mathbb{N}$, $1 \leqslant N^{\ast} \leqslant N$. By Lemma \ref{26}, we have there exist non-negative functions  $\varphi_i \in C_0 ^{\infty}\left(\mathbb{R}^{n}\right)$ with  $\spt \varphi _i \subset\frac{3}{2} Q_i$ such that $\sum_{i=1}^{N} \varphi_i=1$ on $\cup_{i=1}^N Q_i \supset K$. Furthermore,  $\left|\nabla \varphi_i \right| \leqslant C_{10} (n)\ell^{-1}$, $i=1, \ldots, N$, on  $\mathbb{R}^{n}$. Since $u$ is a sub-solution of \eqref{1} in $\Omega \backslash E$, and  $\varphi \varphi _i \in W^{1, p(\cdot)}\left(\Omega ,  \vl \right)$, $\varphi \varphi _i \geqslant 0$, $i=N_\ast +1, \ldots, N$ has compact support in $\Omega \backslash E$, we have
\begin{equation} \label{63}
\int_{\Omega \backslash E}\left\langle \ao (\cdot , u, \nabla u), \nabla\left(\varphi \varphi_i \right)\right\rangle + \bo( \cdot, u)\left(\varphi \varphi_i \right) \dx\leqslant 0,
\end{equation}
Then, by \eqref{4} and \eqref{5},
\begin{align}
&\int_{\Omega} \left\langle \ao (\cdot , u, \nabla u), \nabla \varphi\right\rangle  + \bo ( \cdot , u) \varphi  \dx \nonumber 
=\sum_{i=1}^{N} \int_{\frac{3}{2} Q_i}\left\langle \ao ( \cdot , u, \nabla u), \nabla\left(\varphi \varphi_i \right)\right\rangle   + \bo (\cdot, u) \left(\varphi \varphi_i\right)   \dx \nonumber\\
&\leqslant  \sum_{i=1}^{N_\ast} \int_{\frac{3}{2} Q_i}\left\langle \ao (\cdot , u, \nabla u), \nabla\left(\varphi \varphi_i\right)\right\rangle  + \bo (\cdot , u) \left(\varphi \varphi_i \right) \dx \nonumber \\
&\leqslant  \sum_{i=1}^{N_{\ast}} \int_{\frac{3}{2} Q_i} \mu ^{-1}(2+ C_{10}) \ell ^{-1}\left( |\nabla u|^{p-1} +u^{p-1}  \right) +(1+C_{10})\ell ^{-1} a_2  + a_3\dvt \nonumber\\
&\leqslant  C_{11} (n,\mu)\sum_{i=1}^{N_\ast}  \int_{\frac{3}{2} Q_i} \ell ^{-1}  \left(|\nabla u|^{p} + 1\right)  + \ell ^{-1} \left(u^p + 1\right)  + \ell ^{-1} a_2  + a_3\dvt \nonumber \\  
&=C_{11} \int_{E_{3 \ell \sqrt{n}}} \ell ^{-1}  \left(\sum_{i=1}^{N_{\ast}} \chi _{\frac{3}{2} Q_i} \right) |\nabla u|^p    +  \ell ^{-1} \left(\sum_{i=1}^{N_{\ast}} \chi _{\frac{3}{2} Q_i} \right) u^{p} \dvt \nonumber \\
&+C_{11}  \int_{E_{3 \ell \sqrt{n}}} \ell ^{-1} \left(\sum_{i=1}^{N_{\ast}} \chi _{\frac{3}{2} Q_i} \right)   + \left(\sum_{i=1}^{N_{\ast}} \chi _{\frac{3}{2} Q_i} \right)  \dvt \nonumber\\
&+C_{11}  \int_{E_{3 \ell \sqrt{n}}} \ell ^{-1} \left(\sum_{i=1}^{N_{\ast}} \chi _{\frac{3}{2} Q_i} \right) a_2  + \left(\sum_{i=1}^{N_{\ast}} \chi _{\frac{3}{2} Q_i} \right) a_3\dvt. \label{15} 
\end{align}
Here $\chi _{\frac{3}{2} Q_i}$ is the characteristic function of $\frac{3}{2} Q_i$, $i=1, \ldots, N_\ast$.  For each cube $Q_i$, $i=1, \ldots, N_\ast$, there are at most $3^n$ cubes $Q_j$, $\ell (Q_j) = \ell$,  $j=1, \ldots, N_i \leqslant 3^{n}$ (adjacent cubes to $Q_i$ with equal length), such that $\frac{3}{2} Q_i \cap \frac{3}{2} Q_j \neq \emptyset$. Then
\begin{equation}\label{28}
\sum_{i=1}^{N_{\ast}} \chi_{\frac{3}{2} Q_i} \leqslant 3^n  \quad \text {  on  }  E_{ 3 \ell \sqrt{n}}.
\end{equation}
By \eqref{15}, \eqref{28}, \eqref{14}, \eqref{10}, \eqref{11}, \eqref{31} and \eqref{67}, we have
$$
\begin{aligned}
&\int_{\Omega}\left\langle \ao(\cdot, u, \nabla u), \nabla \varphi \right\rangle  +  \bo (\cdot, u) \varphi  \dx 
\leqslant (C_7+C_2) C_{11}\ell^{-1} 3^{n}(3 \ell \sqrt{n}) \varepsilon ^{\gamma_1} 
+C_{12}\ell^{-1} 3^{n}(3 \ell \sqrt{n})^{n-\alpha ^+} \varepsilon\\
&+ C_{11} \ell ^{-1}  3^n 2 (3 \ell\sqrt{n})^{(n-\alpha^+)\left(\frac{q_{2}^{-}-1}{q_{2}^{-}}\right)}\left\|a_{2}\right\|_{ L^{q_{2} (\cdot)}\left(\Omega, \vl \right)} \varepsilon ^{\gamma_1} 
 + C_{11}   3^n 2 (3 \ell\sqrt{n})^{(n-\alpha^+)\left(\frac{q_3^{-}-1}{q_3^{-}}\right)}\left\|a_3\right\|_{ L^{q_3 (\cdot)}\left(\Omega, \vl \right)} \varepsilon ^{\gamma_1}  \\
 &\leqslant C_{13} \varepsilon ^{\gamma_1},
\end{aligned}
$$
where $C_{12}$, $C_{13}$ are positive constants independents of $\varepsilon$. Since $\varepsilon \in (0,1)$ was arbitrary, it follows that 
$$
\int_{\Omega} \left\langle \ao (\cdot , u, \nabla u), \nabla \varphi\right\rangle  + \bo ( \cdot , u) \varphi  \dx \leqslant 0.
$$
This completes the proof of Theorem \ref{27}.
\end{proof}

\subsection{Singular set satisfying $\mathcal{L} _{\alpha , \gamma , \vl} (E)=0$}  In this section we will work with nonlinear elliptic equations of the form

\begin{equation} \label{60}
-\operatorname{div}  \ao \left( x , \nabla u \right)= 0.
\end{equation}

We assume that $A:\Omega \times \mathbb{R}^n \rightarrow \mathbb{R}^n$ is a measurable function. Furthermore, there exists $\mu >0$ such that the following conditions are satisfied almost everywhere:
\begin{gather}
\left\langle \ao (x,\eta) , \eta \right\rangle \geqslant \mu   |\eta|^{p(x)}  \vl (x), \label{45}\\
\left|\ao (x,\eta)\right| \leqslant  \mu ^{-1} |\eta|^{p(x)-1} \vl (x), \label{46}
\end{gather}
where  $\vl\in W_p (\Omega)$,  $\alpha, \tau :\Omega \rightarrow [0,n]$ are measurable functions, and $p\in L^\infty _+ (\Omega)$. Moreover,
\begin{gather} 
\tau \leqslant n-\alpha - p^+ - p^+/p^- \quad \text { \textit{a.e.} on }  \Omega .\label{30}
\end{gather}

\begin{definition}
We will say that $u \in W^{1, p(\cdot)} _{E,\loc}  \left(\Omega ,  \vl \right)$ is a (weak) sub-solution  of equation \eqref{60} in $\Omega \backslash E$ if  
\begin{equation} \label{61}
\begin{split}
 \int_{\Omega \backslash E } \left\langle \ao (\cdot , \nabla u) , \nabla \varphi \right\rangle  \dx  \leqslant 0
\end{split}
\end{equation}
for all  $\varphi \in W^{1 , p(\cdot)}\left(\Omega , \vl \right)$, $\varphi \geqslant 0$, with compact support in $ \Omega \backslash E$.
\end{definition}

\begin{definition}
We will say that the  sub-solution $u$ of equation \eqref{60} in $\Omega \backslash E$ has a removable singularity at  $E$: if $u \in W^{1, p(\cdot)} _{E, \loc}  \left(\Omega ,  \vl \right)$ implies $u \in W^{1 , p(\cdot)} _{\loc}  \left(\Omega  , \vl \right)$ and the equality \eqref{61} is fulfilled for all $\varphi \in W^{1, p(\cdot)}  \left(\Omega ,  \vl \right)$, $\varphi \geqslant 0$, with compact support.
\end{definition}

\begin{lemma}
Suppose that there is a  locally bounded function $\sigma: \Omega \times \Omega \rightarrow (0,\infty)$   such that
\begin{equation} \label{22}
 \int _{Q (y , \ell )} \ell  ^{-\tau} \dvt \leqslant \sigma   \left(y,z\right) \int _{Q (z , \ell)} \ell ^{-\tau} \dvt,
\end{equation}
for all adjacent dyadic cubes $Q (y , \ell ), Q (z , \ell)\subset \Omega$. If $K\subset \Omega$ is a compact set, then
\begin{gather}
\int _{\frac{3}{2} Q (y , \ell)} \ell  ^{-\tau} \dvt \leqslant C(n , \sigma , K) \int _{Q (z , \ell)} \ell ^{-\tau} \dvt, \label{49}\\
\int _{B(y , \ell )} \ell  ^{-\tau} \dvt \leqslant C(n, \sigma , K) \int _{Q(y , \ell )} \ell ^{-\tau} \dvt  \label{54},
\end{gather}
where $y=z$ or $Q (y , \ell ), Q (z, \ell)\subset \Omega$ are adjacent dyadic cubes,  with $y, z \in K$ and  $\ell \sqrt{n} < \di (K,\partial \Omega)$.
\end{lemma}

\begin{lemma}
Assume that \eqref{45} - \eqref{30}, \eqref{22} are satisfied, and $\mathcal{L} _{\alpha , \gamma , \vl} (E) =0$ for some $\gamma \in (0,1/p^+]$. If $u$ is a non-negative sub-solution of \eqref{60} in $\Omega \backslash E$ and satisfies
\begin{equation}\label{18}
\int _{B(y,r) } u^p \dvt \leqslant C_\ast \int _{B\left(y, r  / \sqrt{n}\right) }  r^{-\tau} \dvt,
\end{equation}
for all ball $B(y,r)\subset \subset \Omega$, where $C_\ast$ is a suitable positive constant. Then 
\begin{equation}\label{23}
\begin{aligned}
&\min \left\{ \left( \int _{B\left(y,r_1\right)} |\widetilde{\nabla u}|^p \dvt \right) ^{\frac{1}{p^-}}, \left( \int _{B\left(y,r_1\right)} |\widetilde{\nabla u}|^p \dvt \right) ^{\frac{1}{p^+}}\right\}\\
& \leqslant C_0 \max\left\{ \left[\left(r_2 - r_1 \right)^{-p^+}\int _{B(y,r_2 / \sqrt{n})}  r_2 ^{-\tau} \dvt \right]^{\frac{1}{p^-}}  ,  \left[\left(r_2 - r_1 \right)^{-p^+}\int _{B(y,r_2 / \sqrt{n})}  r_2 ^{-\tau} \dvt \right]^{\frac{1}{p^+}}   \right\}
\end{aligned}
\end{equation}
for all ball $B(y,r_2)\subset \subset \Omega$ with $0< r_2 - r_1 <1$, where $C_0 = C_0(n ,\mu , p , \sigma , E , C_\ast)>0$.
\end{lemma}
\begin{proof}
1. Let $r_2 >r_1>0$ and $y\in \Omega $ be such that $\overline{B(y,r_2)}\subset \Omega$ and $r_2 - r_1 <1$. Choose $r_0 \in (0,1)$ such that $\overline{E_{r_0}}\subset \Omega$. Take $\varepsilon \in (0,r_0)$ be arbitrarily. Since $\mathcal{L}_{\alpha , \gamma , \vl}(E)=0$ there is a finite collection of mutually disjoint dyadic cubes $Q_i=Q(y_i , \ell _i)$ so that  $E \subset \inter \left(\cup_{i=1}^{N} Q_{i}, \right)$,
\begin{equation} \label{33}
\sum_{i=1}^{N} \left(\int_{Q_{i}} \ell _i ^{\alpha-n} \dvt\right)^\gamma<\varepsilon,
\end{equation}
$3\ell_{i} \sqrt{n}< \varepsilon$ and $Q_i \cap E \neq \emptyset$, $i=1, \ldots, N$.
 
Now, we attach to each cube $Q_{i}$, $\ell\left(Q_{i}\right)=\ell _{i}$, $i=1, \ldots, N$, all adjacent dyadic cubes with the same length $\ell _i$. Since two dyadic cubes are either mutually disjoint or one is contained in the other, we may drop extra cubes away. Proceeding in this way we get a collection of mutually disjoint cubes $Q_{j,k} = Q (y_{j,k} , \ell _{j,k})$,  $(j,k)\in \{1, \ldots, N_1\} \times \{1, \ldots, m_j\}$ satisfying 
\begin{equation} \label{16}
\ell _{j,k}=\ell_{a_j}, \quad  j=1, \ldots, N_1 \leqslant N, \quad k=1, \ldots, m_j \leqslant 3^{n},
\end{equation}
where $\{1\leqslant a_1 < \cdots < a_{N_1} \leqslant N\} \subset \mathbb{Z}$, and furthermore $Q_{j,k}$ is adjacent to $Q_{a_j}$.

Additionally, by \eqref{18} and \eqref{49} ,
\begin{equation} \label{53}
 \int _{ \frac{3}{2}Q_{j,k}}   u^p \dvt  \leqslant C_1  \int _{ Q_{a_{j}}}  \ell _{a_{j}} ^{-\tau} \dvt,
\end{equation}
$j=1, \ldots, N_1$, $k=1, \ldots, m_j$, for some constant $C_1= C_1 \left(n , \sigma , E_{2r_0 /3} , C_\ast\right)>0$.\\

2.  We next choose non-negative test functions $\varphi_{i}$, $\varphi _{j,k}$, $i=1, \ldots N$, $j=1, \ldots, N_1$, $k=1, \ldots, m_{j}$, given by the  Lemma \ref{26}, with the following properties:
\begin{gather}
 \spt \varphi_{i} \subset \frac{3}{2} Q_{i}, \quad \left|\nabla \varphi _{i}\right| \leqslant C_2 (n)\ell _{i} ^{-1}, \label{41} \\
 \spt \varphi _{j,k} \subset \frac{3}{2} Q_{j,k}, \quad \left|\nabla \varphi _{j,k}\right| \leqslant C_2 (n) \ell _{j,k} ^{-1}, \label{42}\\
 \sum_{i}\varphi_{i} + \sum _{j,k}\varphi _{j,k}=1 \quad  \text { on } \left(\cup_{i } Q_{i }\right) \cup \left(\cup _{j,k} Q_{j,k} \right).\label{43}
\end{gather}
Let  $\psi _0: \mathbb{R}^n\rightarrow \mathbb{R}$ be a smooth function with $\spt \psi _0 \subset B(y,r_2)$, $0 \leqslant \psi _0 \leqslant 1$, $|\nabla \psi _0| \leqslant C_3 (n) (r_2 - r_1)^{-1}$, and $\psi=1$ on $B(y,r_1)$. Write
$$
\xi= \psi _0 \left[1- \left( \sum_{i} \varphi_{i} + \sum _{j,k} \varphi _{j,k}\right)\right].
$$
Substituting $\xi ^{p^+} u$ into \eqref{61}, we see
\begin{gather*}
\int_{\Omega \backslash E}\left\langle A(\cdot , \nabla u), \nabla\left(\xi^{p^+} u\right)\right\rangle \dx \leqslant 0,\nonumber\\
\int _{\Omega \backslash E} \xi ^{p^+} \langle A(\cdot , \nabla u) , \nabla u\rangle \dx\leqslant \int _{\Omega \backslash E} p^+ \xi ^{p^+ -1}|u A(\cdot , \nabla u) ||\nabla \xi| \dx. \label{64} 
\end{gather*}
Then, by  \eqref{45}, \eqref{46}, \eqref{31} and \eqref{57},
$$
\begin{aligned}
&\int _{\Omega \backslash E } \xi ^{p^+} |\nabla u|^p \dvt \leqslant \mu ^{-2}\int _{\Omega \backslash E}  p^+ \xi ^{p^+ -1} |\nabla \xi| |\nabla u|^{p-1 } u\dvt \\
& \leqslant 2 \mu ^{-2} p^{+} \left\|\xi ^{p^{+} - 1} |\nabla u| ^{p-1} \vl^{\frac{p-1}{p}} \right\|_{L^{\frac{p}{p-1} (\cdot)} \left( \Omega \backslash E \right)}
 \left\| |\nabla \xi| u \vl^{\frac{1}{p}} \right\|_{L^{p (\cdot)} \left( \Omega\right)}\\
 & \leqslant 2 \mu ^{-2} p^{+} \max\left\{ \left( \int _{\Omega \backslash E} \xi ^{p^+}|\nabla u|^p \dvt \right) ^{\left(\frac{p-1}{p}\right)^+}, \left( \int _{\Omega \backslash E} \xi ^{p^+}|\nabla u|^p \dvt \right) ^{\left(\frac{p-1}{p}\right)^-}\right\}\\
& \cdot \left[   \sum_{i}  \left\| |\nabla \varphi _i| u \vl^{\frac{1}{p}}\right\|_{L^{p(\cdot)}\left( \frac{3}{2}Q_{i} \right)} + \sum _{j,k} \left\| \left|\nabla \varphi _{j,k}\right|u\vl^{\frac{1}{p}}\right\|_{L^{p(\cdot)}\left(\frac{3}{2}Q_{j,k} \right)}    
  + \left\| \left|\nabla \psi _0 \right| u\vl^{\frac{1}{p}} \right\|_{L^{p(\cdot)} \left(B\left(y , r_2\right)\right)} \right].
\end{aligned}
$$
Consequently, using \eqref{57}, \eqref{18}, \eqref{49}, \eqref{53}, \eqref{30} and \eqref{33}, we compute 
\begin{align*}
&\min \left\{ \left( \int _{\Omega \backslash E} \xi ^{p^+}|\nabla u|^p \dvt \right) ^{1-\left(\frac{p-1}{p}\right)^+}, \left( \int _{\Omega \backslash E} \xi ^{p^+}|\nabla u|^p \dvt \right) ^{1-\left(\frac{p-1}{p}\right)^-}\right\}\\
 & \leqslant C_4 (n, \mu , p ) \max\left\{ \left[\left(r_2 - r_1 \right)^{-p^+}\int _{B(y,r_2)} u^p \dvt \right]^{\frac{1}{p^-}}  ,  \left[\left(r_2 - r_1 \right)^{-p^+}\int _{B(y,r_2)}  u^p \dvt \right]^{\frac{1}{p^+}}   \right\}\\
& + C_4\sum _{i}\max\left\{ \left(\ell _{i}^{-p^+}\int _{\frac{3}{2}Q_{i}} u^p \dvt \right)^{\frac{1}{p^-}}  ,  \left(\ell _{i}^{-p^+}\int _{\frac{3}{2}Q_{i}}  u^p \dvt \right)^{\frac{1}{p^+}}   \right\}\\
& + C_4\sum _{j,k}\max\left\{ \left(\ell _{j,k}^{-p^+}\int _{\frac{3}{2}Q_{j,k}} u^p \dvt \right)^{\frac{1}{p^-}}  ,  \left(\ell _{j,k}^{-p^+}\int _{\frac{3}{2}Q_{j,k}}  u^p \dvt \right)^{\frac{1}{p^+}}   \right\}\\
& \leqslant C_5 \max\left\{ \left[\left(r_2 - r_1 \right)^{-p^+}\int _{B(y,r_2 / \sqrt{n})}  r_2 ^{-\tau} \dvt \right]^{\frac{1}{p^-}}  ,  \left[\left(r_2 - r_1 \right)^{-p^+}\int _{B(y,r_2 / \sqrt{n})}  r_2 ^{-\tau} \dvt \right]^{\frac{1}{p^+}}   \right\}
 + C_6 \varepsilon,
\end{align*}
where $C_5 = C_5 \left(n,\mu , p , \sigma , E_{2r_0 /3}  , C_\ast \right)>0$ and  $C_6>0$ is independent of $\varepsilon$. Hence,
\begin{align*}
&\min \left\{ \left( \int _{\Omega \backslash E} \xi ^{p^+} |\nabla u|^p \dvt \right) ^{\frac{1}{p^-}}
, \left( \int _{\Omega \backslash E}\xi ^{p^+}|\nabla u|^p \dvt \right) ^{\frac{1}{p^+}}\right\}\\
& \leqslant C_5 \max\left\{ \left[\left(r_2 - r_1 \right)^{-p^+}\int _{B(y,r_2 / \sqrt{n})}  r_2 ^{-\tau} \dvt \right]^{\frac{1}{p^-}}  ,  \left[\left(r_2 - r_1 \right)^{-p^+}\int _{B(y,r_2 / \sqrt{n})}  r_2 ^{-\tau} \dvt \right]^{\frac{1}{p^+}}   \right\}
+ C_6 \varepsilon.
\end{align*}
Since $\varepsilon \in (0,1)$ was arbitrary:
\begin{align*}
&\min \left\{ \left( \int _{B\left(y,r_1\right)} |\widetilde{\nabla u}|^p \dvt \right) ^{\frac{1}{p^-}}, \left( \int _{B\left(y,r_1\right)} |\widetilde{\nabla u}|^p \dvt \right) ^{\frac{1}{p^+}}\right\}\\
& \leqslant C_5 \max\left\{ \left[\left(r_2 - r_1 \right)^{-p^+}\int _{B(y,r_2 / \sqrt{n})}  r_2 ^{-\tau} \dvt \right]^{\frac{1}{p^-}}  ,  \left[\left(r_2 - r_1 \right)^{-p^+}\int _{B(y,r_2 / \sqrt{n})}  r_2 ^{-\tau} \dvt \right]^{\frac{1}{p^+}}   \right\}.
\end{align*}
Therefore, we conclude \eqref{23}.
\end{proof}

Motivated by Theorem \ref{69} we have our last result.

\begin{theorem} \label{62}
Assume that \eqref{45} - \eqref{30}, \eqref{22} and  \eqref{18} are satisfied, furthermore $\mathcal{L} _{\alpha , \gamma , \vl} (E)=0$ for some $\gamma \in (0, 1 / p^+]$. If $u$ is a non-negative sub-solution of \eqref{60} in $\Omega \backslash E$, then the singularity of $u$ at $E$ is removable. 
\end{theorem}

\begin{proof}

1. Let $\varphi \in W^{1,  p(\cdot)}(\Omega  , \vl)$ be with compact support. By  Assumption \ref{70} we can suppose that $\varphi \in C^{\infty} (\Omega)$ and $|\nabla \varphi|\leqslant 1$. Write $K=\spt \varphi$, choose $r_0 \in (0, 1)$ such that $\overline{K_{r_0}} \subset \Omega$. Let $\varepsilon \in (0,r_{0})$ be arbitrarily given.  Since $\mathcal{L}_{\alpha , \gamma , \vl}(E)=0$,  there is a finite collection of mutually disjoint dyadic cubes $Q_{i}(y_i , \ell _i)$, $3\sqrt{n}\ell_{i}<\varepsilon$, $i=1, \ldots, N$, so that $E \subset \inter  \left(\cup_{i=1}^{N} Q_{i} \right)$ and
\begin{equation} \label{37} 
\sum_{i=1}^{N}\left(  \int_{Q_{i}} \ell _i ^{\alpha-n} \dvt \right) ^\gamma <\varepsilon.
\end{equation}

We will cover $K$ by a suitable collection of mutually disjoint dyadic cubes:

\noindent We attach to each cube $Q_{i}$, $\ell\left(Q_{i}\right)=\ell _{i}$, $i=1, \ldots, N$, all adjacent dyadic cubes with the same length $\ell _i$. Since two dyadic cubes are either mutually disjoint or one is contained in the other, we may drop extra cubes away. Proceeding in this way we get a collection of mutually disjoint cubes $Q_{j,k} = Q(y_{j,k} , \ell _{j,k})$, $(j,k)\in \{1, \ldots, N_1\} \times \{1, \ldots, m_j\}$ satisfying 
\begin{equation} \label{38}
\ell _{j,k}=\ell_{a_j}, \quad  j=1, \ldots, N_1 \leqslant N, \quad k=1, \ldots, m_j \leqslant 3^{n},
\end{equation}
where $\{1\leqslant a_1 < \cdots < a_{N_1} \leqslant N\} \subset \mathbb{Z}$, and furthermore $Q_{j,k}$ is adjacent to $Q_{a_j}$.

\noindent Finally, we cover the remaining set $K \backslash\left[\left(\cup_{i} Q_{i}\right)\cup \left(\cup_{j , k } Q_{j , k }\right) \right]$ by mutually disjoint dyadic cubes $\widetilde{Q}_h$, all with the same length $\ell \left(\widetilde{Q}_h \right)=\ell _0$, $h=1, \ldots, N_2$, where $\ell _0 = \min \left\{\ell _{i}\:|\: i=1, \ldots, N\right\}$. We see
$$
\frac{3}{2} \widetilde{Q}_h \cap E=\emptyset \quad \text { for }   h=1, \ldots, N_2 .
$$\\

2. From  \eqref{23} we have  $|\widetilde{\nabla u}|\in L^{p(\cdot)} _{\loc}(\Omega , \vl)$. We claim $u\in W_{\loc}^{1, p(\cdot)}\left(\Omega, \vl \right)$.  To verify this assertion, let $\psi _i$, $\psi _{j,k}$, $i=1, \ldots, N$, $j=1, ..., N_1$, $k=1, ..., m_j$  be non-negative test functions given by the Lemma \ref{26}, satisfying 
\begin{gather*}
\spt \psi_{i} \subset \frac{3}{2} Q_{i}, \quad \left|\nabla \psi _{i}\right| \leqslant C_1 (n) \ell _{i} ^{-1},  \\
\spt \psi_{j,k} \subset \frac{3}{2} Q_{j,k}, \quad \left|\nabla \psi _{j,k}\right| \leqslant C_1 (n) \ell _{j,k} ^{-1},  \\
\sum_{i}\psi_{i} + \sum_{j,k}\psi_{j,k}  =1 \quad  \text {  on }   \left(\cup _i Q_i \right) \cup\left(\cup _{j,k} Q_{j,k} \right).
\end{gather*}
Write 
$$
\zeta = \sum_{i}\psi_{i} + \sum_{j,k}\psi_{j,k}. 
$$
We have 
$$
\int _{\Omega \backslash E} u \left[\left(1-\zeta \right)\varphi  \right] _{x_b}\dx=- \int _{\Omega \backslash E} u_{x_b}\left(1-\zeta \right)\varphi  \dx,
$$
where $b=1, \ldots, n$. If $\varepsilon\rightarrow 0^+$:
\begin{equation}\label{59} 
\int _{\Omega} u \varphi _{x_b}\dx=- \int _{\Omega } u_{x_b}\varphi  \dx.
\end{equation}
Indeed, \eqref{31}, \eqref{32}, \eqref{18}, \eqref{49}, \eqref{30} and \eqref{37} implies
\begin{align*}
&\left|\int _{\Omega \backslash E} u \zeta _{x_b}\varphi  \dx \right|\leqslant 2 \left\|u\zeta _{x_b} \vl ^{1/p}\right\| _{L^{p(\cdot)} \left( \Omega\right)} \left\|\vl ^{-1/p}\right\|_{L^{p^{\prime}(\cdot)}  \left( \Omega\right)} \\
&\leqslant C_2 (\vl)\sum _i \left\|u\left|\nabla \psi _{i}\right| \vl ^{1/p}\right\| _{L^{p(\cdot)} \left( \frac{3}{2} Q_{i}\right)} + C_2 \sum _{j,k} \left\|u\left|\nabla \psi _{j,k}\right| \vl ^{1/p}\right\| _{L^{p(\cdot)} \left( \frac{3}{2} Q_{j,k}\right)}\\
& \leqslant C_3 (\vl , n) \sum _i \max \left\{\left(\ell _{i} ^{-p^+} \int _{\frac{3}{2}Q_{i}} u^p \dvt \right)^{\frac{1}{p^-}}    ,   \left(\ell _{i} ^{-p^+}\int _{\frac{3}{2}Q_{i}}  u^p \dvt \right)^{\frac{1}{p^+}} \right\}\\
& + C_3 \sum _{j,k} \max \left\{\left(\ell _{j,k} ^{-p^+} \int _{\frac{3}{2}Q_{j,k}} u^p \dvt \right)^{\frac{1}{p^-}}    ,   \left(\ell _{j,k} ^{-p^+}\int _{\frac{3}{2}Q_{j,k}}  u^p \dvt \right)^{\frac{1}{p^+}} \right\}\\
& \leqslant C_4 (\vl , n , \sigma , K) \varepsilon \rightarrow 0 \quad \text { as }  \varepsilon \rightarrow 0^+.
\end{align*}
So \eqref{59} is proved. \\

3. Additionally we assume that $0\leqslant \varphi \leqslant 1$. We show
\begin{equation*}
\int_{\Omega}\left\langle \ao \left(\cdot, \nabla u\right), \nabla\varphi \right\rangle \dx\leqslant 0.
\end{equation*}
We next choose non-negative test functions $\varphi_{i}$, $\varphi_{j,k}$, $\widetilde{\varphi} _h$, $i=1, \ldots N$, $j=1, \ldots N_1$, $k=1, \ldots, m_j$,  $h=1, \ldots, N_{2}$, given by the  Lemma \ref{26}, with the following properties:
\begin{gather*}
\spt \varphi_{i} \subset \frac{3}{2} Q_{i}, \quad \left|\nabla \varphi _{i}\right| \leqslant C_{5} (n) \ell _{i} ^{-1},  \\
\spt \varphi_{j,k} \subset \frac{3}{2} Q_{j,k}, \quad \left|\nabla \varphi _{j,k}\right| \leqslant C_{5} (n) \ell _{j,k} ^{-1},  \\
  \spt \widetilde{\varphi} _h \subset \frac{3}{2} \widetilde{Q}_h, \quad \left|\nabla \widetilde{\varphi}_{h}\right| \leqslant C_{5} (n) \ell _0 ^{-1},\\
\sum_{i}\varphi_{i} + \sum_{j,k}\varphi_{j,k} + \sum _h \widetilde{\varphi}_h =1 \quad  \text {  on }  K.
\end{gather*}
Write
$$
\xi= \sum_{i}\varphi_{i} + \sum_{j,k}\varphi_{j,k} + \sum _h \widetilde{\varphi}_h.
$$
Since $ \frac{3}{2} \widetilde{Q}_h \cap E = \emptyset$ and  $\spt \widetilde{\varphi} _h \subset \frac{3}{2} \widetilde{Q}_h$, substituting $\varphi \widetilde{\varphi}_{h}$ into \eqref{61}, we have 
$$
\int_{\frac{3}{2} \widetilde{Q}_{h}}\left\langle \ao \left(\cdot,  \nabla u\right), \nabla\left(\varphi \widetilde{\varphi}_{h}\right)\right\rangle  \dx\leqslant 0,
$$
for $h=1, \ldots, N_2$. From this, 
\begin{equation} \label{65}
\begin{aligned}
&\int_{\Omega}\left\langle \ao \left(\cdot,  \nabla u\right), \nabla\varphi \right\rangle \dx
= \int_{\Omega}\left\langle \ao \left(\cdot ,  \nabla u\right), \nabla\left(\varphi \xi\right)\right\rangle \dx \\
&\leqslant  \sum _{i}\int_{\frac{3}{2}Q_{i}} \left\langle \ao \left(\cdot, \nabla u\right), \nabla\left( \varphi \varphi _{i} \right) \right\rangle    \dx 
+ \sum _{j,k}\int_{\frac{3}{2}Q_{j,k}} \left\langle \ao \left(\cdot,  \nabla u\right), \nabla\left( \varphi \varphi _{j,k} \right) \right\rangle  \dx .
\end{aligned}
\end{equation}
We can assume that $\int _{E_{r_0}} |\nabla u|^p \dvt <1$, since $|E| =0$. By \eqref{65}, \eqref{46}, \eqref{49}, \eqref{23} and \eqref{37},
\begin{align}
&\int_{\Omega}\left\langle \ao \left(\cdot,  \nabla u\right), \nabla\varphi \right\rangle  \dx\nonumber 
\leqslant C_{6} \sum _i \int_{\frac{3}{2}Q_{i}} \ell _i^{-1} \left(|\nabla u|^p +1\right)\dvt 
+ C_{6} \sum _{j,k} \int_{\frac{3}{2}Q_{j,k}} \ell _{j,k}^{-1} \left(|\nabla u|^p +1\right)\dvt \nonumber\\
&\leqslant  C_{7} \varepsilon + C_{6} \sum _i \int_{B\left(y_i , 3\sqrt{n} \ell _i / 4 \right)} \ell _i^{-1} |\nabla u|^p \dvt + C_{6} \sum _{j,k}  \int_{B\left(y_{j,k }, 3\sqrt{n} \ell _{j,k}/ 4 \right)}  \ell _{j,k}^{-1} |\nabla u|^p \dvt \nonumber\\
&\leqslant  C_{7} \varepsilon + C_8 \sum _i \ell_{i} ^{-1} \max \left\{  \left[\left(\sqrt{n} \ell _{i} - \frac{3\sqrt{n} \ell _{i}}{4} \right)^{-p^+}\int_{B\left(y_i ,  \ell _i  \right)} \left(\sqrt{n} \ell _{i} \right)^{-\tau}\dvt\right]  \right. \nonumber \\
&\left. , \left[\left(\sqrt{n} \ell _{i} - \frac{3\sqrt{n} \ell _{i}}{4} \right)^{-p^+}\int_{B\left(y_i ,  \ell _i  \right)} \left(\sqrt{n} \ell _{i} \right)^{-\tau}\dvt\right]^{\frac{p^-}{p^+}}  \right\} \nonumber\\
&+  C_8 \sum _{j,k} \ell_{j,k} ^{-1} \max \left\{  \left[\left(\sqrt{n} \ell _{j,k} - \frac{3\sqrt{n} \ell _{j,k}}{4} \right)^{-p^+}\int_{B\left(y_{j,k} ,  \ell _{j,k}  \right)} \left(\sqrt{n} \ell _{j,k} \right)^{-\tau}\dvt\right]  \right. \nonumber \\
&\left. , \left[\left(\sqrt{n} \ell _{j,k} - \frac{3\sqrt{n} \ell _{j,k}}{4} \right)^{-p^+}\int_{B\left(y_{j,k} ,  \ell _{j,k}  \right)} \left(\sqrt{n} \ell _{j,k} \right)^{-\tau}\dvt\right]^{\frac{p^-}{p^+}} \right\}. \label{34}
\end{align}
Combining \eqref{34}, \eqref{54}, \eqref{30} and \eqref{37}, we deduce that
\begin{equation*}
\int_{\Omega}\left\langle \ao \left(\cdot, \nabla u\right), \nabla\varphi \right\rangle \dx\leqslant C_{9} \varepsilon,
\end{equation*}
here the $C_i$, $i=6, \ldots, 9$ are positive constants independents of $\varepsilon$. Since $\varepsilon \in (0,r_0)$ was arbitrary, the proof is complete. 

\end{proof}

\bibliographystyle{plain}
\bibliography{ref}	


\end{document}